\documentclass[graybox]{svmult}

\usepackage{type1cm}        
\usepackage{makeidx}         
\usepackage{graphicx}        
\usepackage{multicol}        
\usepackage[bottom]{footmisc}

\usepackage{newtxtext}       %
\usepackage[varvw]{newtxmath}       

\makeindex             

\newcommand{\cE}{\mathcal{E}}

\newcommand{\cK}{\mathcal{K}}

\newcommand{\cV}{\mathcal{V}}

\newcommand{\cX}{\mathcal{X}}

\newcommand{\bC}{\mathbb{C}}

\newcommand{\bE}{\mathbb{E}}

\newcommand{\bR}{\mathbb{R}}

\newcommand{\be}{\begin{equation}}
\newcommand{\ee}{\end{equation}}
\newcommand{\bal}{\begin{align}}
\newcommand{\eal}{\end{align}}
\newcommand{\ba}{\begin{align*}}
\newcommand{\ea}{\end{align*}}
\newcommand{\bmx}{\begin{matrix}}
\newcommand{\emx}{\end{matrix}}
\newcommand{\bbmx}{\begin{bmatrix}}
\newcommand{\ebmx}{\end{bmatrix}}
\newcommand{\bpmx}{\begin{pmatrix}}
\newcommand{\epmx}{\end{pmatrix}}
\newcommand{\bvmx}{\begin{vmatrix}}
\newcommand{\evmx}{\end{vmatrix}}

\newcommand{\ul}{\underline}

\newcommand{\wh}{\widehat}
\newcommand{\wt}{\widetilde}
\newcommand{\f}{\frac}

\newcommand{\inc}{\subseteq}

\newcommand{\Id}{\mathrm{Id}}

\newcommand{\tr}{\mathrm{tr}}

\newcommand{\argmin}{{\rm argmin}\,}

\newcommand{\minimize}[1]{\underset{#1}{\rm minimize}\,}

\newcommand{\la}{\lambda}
\newcommand{\La}{\Lambda}
\newcommand{\eps}{\varepsilon}
  
\begin{document}

\title*{S-Procedure Relaxation:\\ a Case of Exactness
Involving Chebyshev Centers}
\author{Simon Foucart and Chunyang Liao}
\institute{Simon Foucart \at Texas A\&M University,  College Station\\\email{foucart@tamu.edu}
\and Chunyang Liao \at University of California, Los Angeles\\ \email{liaochunyang@math.ucla.edu}}
%
%
\maketitle



\abstract{Optimal recovery is a mathematical framework for learning functions from observational data by adopting a worst-case perspective tied to model assumptions on the functions to be learned.
Working in a finite-dimensional Hilbert space,
we consider model assumptions based on approximability and observation inaccuracies
modeled as additive errors bounded in $\ell_2$. 
We focus on the local recovery problem, 
which amounts to the determination of Chebyshev centers. 
Earlier work by Beck and Eldar presented a semidefinite recipe for the determination of Chebyshev centers.
The result was valid in the complex setting only, 
but not necessarily in the real setting, 
since it relied on the S-procedure with two quadratic constraints, which offers a tight relaxation only in the complex setting. 
Our contribution consists in proving that this semidefinite recipe is exact in the real setting, too,
at least in the particular instance 
where the quadratic constraints involve orthogonal projectors. 
Our argument exploits a previous work of ours, where exact Chebyshev centers were obtained in a different way.
We conclude by stating some open questions and by commenting on other recent results in optimal recovery.}

\section{Rundown on Optimal Recovery}

The field of optimal recovery,
arguably shaped by the influence of Kolmogorov \cite{Kol}
and deeply rooted in Approximation Theory \cite{MicRiv},
is experiencing a revival thanks to newly available optimization tools.
Chosen as one of the core topics first investigated at CAMDA---the Center for Approximation and Mathematical Data Analytics---optimal recovery can be understood as a nonstatistical learning theory.
Indeed, a function $f$ acquired through point evaluations $y_i = f(x^{(i)})$, $i=1,\ldots,m$,
needs to be learned---or recovered, in the parlance preferred here.
But one does not abide by the postulate behind statistical learning theory,
stipulating that the $x^{(i)}$'s are independent realizations of a random variable. 
Thus, the performance of a learning/recovery procedure cannot be assessed in an average case
and one opts for an assessment focusing on the worst case,
 relative to a model expressing some prior scientific knowledge about $f$.

To be more precise,  
the task at hand consists in recovering an element $f$ from a Banach space $F$.
It is typically thought of as a space of functions,  although it does not have to be.
Rather than recovering $f$ in full,
it can be more relevant to recover a quantity of interest $Q(f)$,
where $Q: F \to Z$ represents a linear map in this article.
The element $f$ is only available through partial information,  specifically:
\begin{itemize}
\item some {\em a priori} information conveyed by a modeling assumption taking the form 
$$
	f \in \cK,
$$
where the so-called model set $\cK \inc F$ reflects an educated guess about realistic objects to be recovered;
\item some {\em a posteriori} information obtained through observational data of the form
$$
y_i = \la_i(f),
\qquad i=1,\ldots,m,
$$ 
for some linear functionals $\la_1,\ldots,\la_m: F \to \bR$.
More concisely, one writes $y = \La f \in \bR^m$,
where the so-called observation map $\La: F \to \bR^m$ is a linear map.
\end{itemize}
A recovery procedure is a process, perhaps partially cognizant of the model set $\cK$,
that takes in the observational data in $\bR^m$ and returns an estimation to $Q(f)$ in $Z$.
In other words, it is nothing but a map $\Delta: \bR^m \to Z$.
Its recovery performance can evidently be quantified
by $\|Q(f)-\Delta(y)\|_Z = \|Q(f)-\Delta(\La f)\|_Z$
for a fixed  $f \in \cK$ satisfying $\La f = y$.
However,
$f$ being unknown, one takes a worst-case perspective over all consistent $f$'s,
leading to a recovery performance quantified by
\begin{itemize}
\item the local worst-case error, at a fixed $y \in \bR^m$, defined as
\be
\label{lwce}
{\rm lwce}(\Delta,y) = \sup_{\substack{f \in \cK \\ \La f = y}} \|Q(f) - \Delta(y) \|_Z;
\ee
\item the global worst-case error defined as
\be
\label{gwce}
{\rm gwce}(\Delta) = \sup_{f \in \cK} \|Q(f) - \Delta(\La f) \|_Z.
\ee
\end{itemize}
A locally, respectively globally, optimal recovery map $\Delta^{\rm opt}: \bR^m \to Z$ is a map $\Delta: \bR^m \to Z$ that minimizes ${\rm lwce}(\Delta,y)$ at every $y \in \bR^m$, respectively  ${\rm gwce}(\Delta)$.
Noticing that ${\rm gwce}(\Delta) = \sup \{ {\rm lwce}(\Delta,y), y \in \La(\cK) \}$,
one realizes that a locally optimal recovery map is automatically globally optimal.
This somewhat makes the global setting `easier' than the local setting,
conceivably explaining the prevalence of the latter in the standard theory of optimal recovery.
There,  a coveted result often consists of the assertion that there exists an optimal recovery map which is linear---of course,  one strives to construct it!
In the global setting,
this existence result classically holds when the model set $\cK$ is symmetric and convex
and $Q$ is a linear functional
(see \cite[Theorem~4.7]{NovWoz} or \cite[Theorem~9.3]{BookDS})
and when the model set is a centered hyperellipsoid in a Hilbert space	
(see \cite[Theorem~4.11]{NovWoz} or \cite[Theorem~9.4]{BookDS}).
In the latter situation, the optimal recovery map, dubbed spline algorithm,  is also locally optimal.
There are other situations where global optimality via linear maps stands,
for instance,  in the space $C(\cX)$ of continuous functions on a compact space $\cX$ 
relative to the model set
\be
\label{ApproxSet}
\cK_\cV = \{ f \in C(\cX): {\rm dist}_{C(\cX)}(f,\cV) \le \eps \} 
\ee
subordinate to a linear subspace $\cV$ of $C(\cX)$ and an approximability parameter~$\eps \ge 0$.
More details will be given in Section \ref{SecRes}.

\section{Our Contribution: Local Optimality from Inaccurate Data}
\label{SecContri}

From now on, we leave the global setting behind and tackle the harder local setting,
starting by highlighting its geometric interpretation.
Namely, considering a locally optimal recovery map $\Delta^{\rm opt}: \bR^m \to Z$
and a fixed $y \in \bR^m$,
since $\Delta^{\rm opt}(y) \in Z$ minimizes $\sup\{ \|Q(f) - z\|_Z: f \in \cK, \La f = y\}$,
one can write, almost tautologically,
that
$$
\Delta^{\rm opt}(y) \in \underset{{z \in Z, r \ge 0}}{\argmin \;} \;\; r
\qquad \mbox{s.to } \quad \|Q(f)-z\|_Z \le r \mbox{ whenever } f \in \cK \mbox{ and } \La f = y.
$$
This shows that $\Delta^{\rm opt}(y)$ is a center of a smallest-radius ball containing $Q(\cK_y)$,
where $\cK_y := \{ f \in F: f \in \cK \mbox{ and } \La f = y \}$.
It is said that $\Delta^{\rm opt}(y)$ is a Chebyshev center for the set $Q(\cK_y)$,
often eluding to mention the norm on $Z$.

But the above context is not quite where our the current investigations take place.
Indeed, the discussion so far assumed that the observational data were accurate.
In realistic situations,  they are contaminated by additive noise, so that $y_i = \la_i(f) + e_i$, $i=1,\ldots,m$.
In short, we write $y = \La f + e$
for some error vector $e \in \bR^m$.
We shall model this vector deterministically through $e \in \cE$ for a so-called uncertainty set $\cE \inc \bR^m$.
Thus, the {\em a~posteriori} information now takes the form
$$
y - \La f \in \cE.
$$ 
This leads to an updated notion of local worst-case error,  as defined by 
$$
{\rm lwce}(\Delta,y) = \sup_{\substack{f \in \cK \\ y - \La f \in \cE}} \|Q(f) - \Delta(y)\|_Z.
$$
Our objective of finding an optimal recovery map $\Delta^{\rm opt}: \bR^m \to Z$ now becomes the determination,
for each $y \in \bR^m$,
of a solution $\Delta^{\rm opt}(y)$ to the optimization program
$$
\minimize{z \in Z} \;  \sup_{\substack{f \in \cK \\ y - \La f \in \cE}} \|Q(f) -z \|_Z.
$$ 
As before, 
one interprets geometrically $\Delta^{\rm opt}(y)$ as a Chebyshev center for the set $Q(\cK_{y,\cE})$,
where $\cK_{y,\cE} = \{ f \in F: f \in \cK \mbox{ and } y - \La f \in \cE \}$,
i.e., as a center for a smallest-radius ball containing $Q(\cK_{y,\cE})$.

To achieve our objective---at least partially---we place ourselves in a Hilbert framework,
i.e., we assume from now on that $F$ is a finite-dimensional Hilbert space,
hence it is denoted by $H$ instead of a generic $F$.
Let us state our first contribution,
before placing it in the context of the current knowledge.
Note that we can safely talk about {\em the} Chebyshev center in this statement,
as it was known as early as \cite{Gar} that a bounded set in a uniformly convex Banach space possesses a unique Chebyshev center.

\begin{theorem}
\label{ThmMain}
In a finite-dimensional Hilbert space $H$,
consider a model set $\cK$ and an uncertainty set $\cE$ given by
$$
\cK  = \{ f \in H: \|Pf\|_H \le \eps\}
\qquad \mbox{and} \qquad
\cE = \{ e \in \bR^m: \|e\|_2 \le \eps \}.
$$
If $P$ is an orthogonal projection, if $\La \La^* = \Id_{m}$,
and if $\ker(P) \cap \ker(\La) = \{0\}$,
then, for any $y \in \bR^m$,
the Chebyshev radius of $\cK_{y,\cE} = \{ f \in H: f \in \cK \mbox{ and } y - \La f \in \cE\}$	
is equal to the optimal value of the semidefinite program
\begin{align}
\label{SDPCheCen}
\minimize{\substack{c,d \ge 0 \\ t \in \bR}}  \;
c \eps^2 + d \eta^2 - d \|y\|_2^2 + t
&  & \mbox{s.to }  & \quad
c P + d \La^* \La  \succeq \Id_H,\\
\nonumber
&  & \mbox{and } & \quad
\bbmx
c P + d \La^* \La & \vline &  d \La^* y\\
\hline
 d y^* \La  & \vline & t
\ebmx
\succeq 0.
\end{align}
Moreover, the Chebyshev center is the solution to a regularization program with specified parameters,
namely it is given by
\be
\label{RegSpe}
f_{\wt{c},\wt{d}} := \underset{f \in H}{\argmin} \;  \wt{c} \, \|P f \|_H^2 + \wt{d} \, \|y - \La f\|_2^2,
\ee
where $\wt{c},\wt{d} \ge 0$ solve the semidefinite program \eqref{SDPCheCen}.
\end{theorem}

For the sake of the following discussion, we select the model and uncertainty sets as arbitrary hyperellipsoids.
Specifically, with $R,S$ representing linear maps from $H$ into other Hilbert spaces 
(all norms now being written as $\|\cdot\|$ for ease of notation)
and with $\eps,\eta$ representing positive parameters,
we consider
\begin{align}
\label{ModelSet}
\cK & = \{ f \in H: \|R f\| \le \eps\} \inc H,\\
\label{UncerSet}
\cE & = \{ e \in \bR^m: \|S e \| \le \eta\} \inc \bR^m.
\end{align} 
In this case,
Micchelli and Melkman \cite{MelMic,Mic} already observed that some regularization (with unspecified parameters) provides linear recovery maps that are optimal,
albeit globally.
Locally, a similar conclusion was derived by Beck and Eldar \cite{BecEld} for the full estimation problem, i.e., when $Q = \Id_H$.
In other words, they---almost---established Theorem~\ref{ThmMain} in the more general situation of the model and uncertainty sets  \eqref{ModelSet}-\eqref{UncerSet}.
There is a subtlety, and an important one: technically, 
the Chebyshev center was determined in the complex setting only,
but not necessarily in the real setting.
This incongruity occurs because the main tool used in the argument,
i.e., the S-procedure with two constraints, is only exact in the complex setting,  see next section for details.
In the real setting,
the S-procedure is only a relaxation,
merely leading to an overestimation of the Chebyshev radius 
rather than to the genuine Chebyshev radius.
Our main contribution in this article therefore consists in showing that the supposed overestimation
actually agrees with the Chebyshev radius (i.e., the minimal local worst-case error),
at least in a particular instance.

The basis of the argument is the fact that the exact Chebyshev radius and center have already been obtained,
in a different form,
in case $P$ is an orthogonal projection and $\La \La^* = \Id_{m}$.
Indeed, as an extension to the local optimal recovery problem in Hilbert spaces
with approximability model \eqref{ApproxSet} and accurate data 
(i.e., with $P = P_{\cV^\perp}$ and $\eta=0$),
which was settled in \cite{BCDDPW},  
we established in \cite[Theorem 8]{FouLia} that the Chebyshev center arises from
 a regularization program with explicitly described parameter $\tau_\sharp \in (0,1)$.
Precisely,
for $\tau \in (0,1)$,
let us define
\be
\label{RegSpe2}
f_\tau := \underset{f \in H}{\argmin} \; (1-\tau) \|P f \|^2 + \tau \|y - \La f\|^2
\ee
as the solution to a regularization program akin to \eqref{RegSpe}.
The Chebyshev center is $f_{\tau_\sharp}$,
where the desired parameter $\tau_\sharp$
is the unique $\tau$ between $1/2$ and $\eps / (\eps + \eta)$ satisfying the implicit equation
\be
\label{ImpEq}
\la_{\min} ((1-\tau)P + \tau \La^* \La) = \f{(1-\tau)^2 \eps^2 - \tau^2 \eta^2}{(1-\tau) \eps^2 - \tau \eta^2 + (1-\tau)\tau(1-2\tau) \delta^2},
\ee
in which $\delta$ is precomputed as $\delta = \min \{ \|Pf \|: \La f = y\} = \min \{ \|\La f - y \|: Pf = 0  \}$.
As explained in \cite[Appendix]{FouLia},
the above equation can be solved efficiently via the Newton method.
According to the yet-to-be-established Theorem \ref{ThmMain},
the Chebyshev center can alternatively be determined by solving the semidefinite program \eqref{SDPCheCen}.
We have not seriously compared these two options,
but we would instinctively favor solving \eqref{ImpEq} to bypass the black-box nature of semidefinite solvers.

\section{Overestimate of the Chebyshev Radius via the S-Procedure}

Our goal in this section is to extend a result of \cite{BecEld} to a quantity of interest $Q$ which is an arbitrary linear map between two Hilbert spaces,
instead of just $Q=\Id_H$. 
This extension is not really difficult,
but our arguments differ slightly from the ones of~\cite{BecEld}.
The result itself,
which provides an upper bound for the Chebyshev radius of $Q(\cK_{y,\cE})$,
as well a candidate for its Chebyshev center,
is based on the S-prodecure relaxation.
When this relaxation is exact,
the upper bound agrees with the Chebyshev radius and the candidate is the genuine Chebyshev center.
It is in the next section that we  establish the exactness of the S-procedure relaxation in our particular instance.
Here, we simply explain where the upper bound is coming from.

To this end, we start by recalling the gist of the S-procedure
and point to the survey \cite{PolTer} for more details.
Given $K+1$ quadratic functions defined on $H$, say
$$
q_k(h) = \langle A_k h, h \rangle + 2 \langle a_k, h \rangle + \alpha_k,
\qquad k=0,1,\ldots,K,
$$
where the $A_k$'s are self-adjoint operators, 
the $a_k$'s are vectors, and the $\alpha_k$'s are scalars,
we consider the two assertions
\begin{align}
\label{SProdAss1}
q_0(h) \le 0 \; \;
& \hspace{-0.5mm}  \mbox{whenever} \; \; 
q_1(h) \le 0, \ldots, q_K(h) \le 0,\\
\label{SProdAss2}
\mbox{there exist }c_1,\ldots,c_K \ge 0 \colon  
& \hspace{-0.5mm} q_0(h) \le c_1 q_1(h) + \cdots + c_K q_K(h) \,  \mbox{ for all }h \in H.
\end{align}
Obviously, if assertion \eqref{SProdAss2} holds, then assertion \eqref{SProdAss1} holds as well.
This, in essence, is the S-procedure.
The question of its exactness is whether \eqref{SProdAss1} and \eqref{SProdAss2} are in reality equivalent.
It is the case for $K=1$: this is Yakubovich S-lemma \cite{Yak}.
We are actually interested in $K=2$ here.
In this situation,  it was shown in \cite{BecEld2} that the S-procedure is exact when the scalar field is $\bC$,
but not necessarily when the scalar field is $\bR$,
which is our primary concern.
Still, under mild assumptions,
exactness holds for $\bR$ and $K=2$ in the absence of linear terms, i.e.,  when $a_0=a_1=a_2=0$.
The latter result, established by Polyak in \cite{Pol}, was the key for us
to settle the global optimality problem in \cite{FouLia}.
But in general, we make do with the mere relaxation: this leads to the overestimation derived below
(which turns into an exact evaluation if the scalar field is $\bC$).

\begin{theorem}
\label{ThmQ}
Let $Q:H \to Z$ be a linear map between finite-dimensional Hilbert spaces.
Given Hilbert-space-valued linear maps $R$ and $S$ defined on $H$ and $\bR^m$, respectively,
and satisfying $\ker(R) \cap \ker(S \La) = \{0\}$,
consider the model set \eqref{ModelSet} and uncertainty set \eqref{UncerSet}, i.e.,
$$
\cK  = \{ f \in H: \|R f\| \le \eps\} \inc H
\qquad \mbox{and} \qquad
\cE = \{ e \in \bR^m: \|S e \| \le \eta\} \inc \bR^m.
$$
For $y \in \bR^m$,
consider $\wt{\gamma}$ and $\wt{c},\wt{d},\wt{t}$ to be the minimal value and minimizers of the semidefinite program
\begin{align}
\label{SDPCheCenQ}
\minimize{\substack{c,d \ge 0 \\ t \in \bR}}  \;
c \eps^2 + d \eta^2 - d \| Sy\|_2^2 + t
&  & \mbox{s.to }  & \quad
c R^* R + d \La^* S^* S \La  \succeq Q^* Q,\\
\nonumber
&  & \mbox{and } & \quad
\bbmx
c R^* R + d \La^* S^* S \La & \vline &  d \La^* S^* S y\\
\hline
 d y^* S^* S \La  & \vline & t
\ebmx
\succeq 0.
\end{align}
Then the Chebyshev radius of $Q(\cK_{y,\cE})$,
$\cK_{y,\cE} := \{ f \in H: f \in \cK \mbox{ and } y - \La f \in \cE \}$, 
is upper bounded as
\be
\label{OverEst}
\inf_{z \in Z} \;  \sup_{\substack{f \in \cK \\ y - \La f \in \cE}} \|Q(f) -z \|
\; \le \; \sup_{\substack{f \in \cK \\ y - \La f \in \cE}} \|Q(f) -Q(f_{\wt{c},\wt{d}}) \|
\; = \; \wt{\gamma}^{1/2} ,
\ee
where $f_{\wt{c},\wt{d}}$ is the solution to a regularization program with parameters $\wt{c},\wt{d}$,
namely 
\be
\label{RegSpeQ}
f_{\wt{c},\wt{d}} := \underset{f \in H}{\argmin} \;  \wt{c} \, \|R f \|^2 + \wt{d} \, \|S(y - \La f)\|^2.
\ee
\end{theorem}

\begin{proof}
We drop the dependence on $y$ throughout the argument below.
For $c,d \ge 0$,  we also use the notation 
$$
f_{c,d} := \underset{f \in H}{\argmin} \;  c \, \|R f \|^2 + d \, \|S(y - \La f)\|^2.
$$
It is essential to keep in mind that $f_{c,d}$ is characterized by
\be
\label{CharaReg}
c R^* R f_{c,d} + d \La^* S^* S (\La f_{c,d} - y) = 0.
\ee
As $c R^* R + d \La^* S^* S \La$ is invertible thanks to the assumption $\ker(R) \cap \ker(S \La) = \{0\}$,
this characterization can be rewritten as 
\be
\label{ExprReg}
f_{c,d} = [c R^* R + d \La^* S^* S \La]^{-1}(d \La^* S^* S y).
\ee
Moreover, working preferentially with squared norms, 
we introduce the quantities 
(${\sf sv}$ is for squared value,
${\sf ub}$ is for upper bound,
and ${\sf lub}$ is for least upper bound):
\begin{align*}
{\sf sv}(z) & =  \sup_{f \in H} \; \|Q(f) -z \|^2
& & \mbox{s.to } \quad f \in \cK \mbox{ and  } y - \La f \in \cE,\\
{\sf ub}(z) & =  \inf_{\gamma, c, d\ge 0}  \;  \gamma
& &  \mbox{s.to} \quad \mbox{the constraint elucidated in \eqref{CstUB} below},\\
{\sf lub} & =
\inf_{c,d \ge 0} \;  {\sf Obj}
&  & \mbox{s.to} \quad
\mbox{the constraint elucidated in \eqref{CstLUB} below},
\end{align*}
where the objective function is ${\sf Obj} = c( \eps^2 -  \|Rf_{c,d}\|^2  ) + d ( \eta^2 - \|S(y - \La f_{c,d})\|^2  )$.
As for the constraints, they read
\begin{align}
\nonumber
 \gamma  - & c( \eps^2 -  \|Rf_{c,d}\|^2  ) - d ( \eta^2 - \|S(y - \La f_{c,d})\|^2  )
 + c \|Rh \|^2 + d \|S \La h \|^2 - \|Qh\|^2\\
 \label{CstUB}
 & \ge  \|Q(f_{c,d}) -z \|^2 + 2 \langle Q(f_{c,d}) -z, Qh \rangle
 \qquad \quad \; \; \,  \mbox{for all }  h \in H,
\end{align}
and
\be
\label{CstLUB}
c \|Rh \|^2 + d \|S \La h \|^2 - \|Qh\|^2 \ge 0 
\qquad \mbox{for all } h \in H.
\ee
We now divide the argument into the proofs of several facts, namely:
\begin{enumerate}
\item[(i)] for all $z \in Z$, ${\sf sv}(z) \le {\sf ub}(z)$;
\item[(ii)] $\inf_{z \in Z} {\sf ub}(z) = {\sf lub} = {\sf ub}(Q(f_{\wt{c},\wt{d}}) )$;
\item[(iii)] the program defining ${\sf lub}$ is equivalent to the program \eqref{SDPCheCenQ}.
\end{enumerate}
Once all these facts are justified, we are able to conclude via
$$
\inf_{z \in Z} {\sf sv}(z)
\le  {\sf sv}(Q(f_{\wt{c},\wt{d}}))
\underset{\rm (i)}{\le}{\sf ub}(Q(f_{\wt{c},\wt{d}}))
\underset{\rm (ii)} = {\sf lub}
\underset{\rm (iii)} = \wt{\gamma},
$$
which,
up to taking square roots,
is the result announced in \eqref{OverEst}.

\vspace{5mm}
\noindent
\ul{Justification of (i).}
For $z \in Z$,
the definition of ${\sf sv}(z)$ specified to our situation yields
\begin{align*}
 {\sf sv}(z)
 & = \sup_{f \in H} \{  \|Q(f) -z \|^2 : \|Rf\|^2 \le \eps^2, \|S(y - \La f)\|^2 \le \eta^2  \}\\
 & = \inf_{\gamma \ge 0} \{ \gamma :
 \|Q(f) -z \|^2  \le \gamma  \mbox{ whenever }  \|Rf\|^2 \le \eps^2 \mbox{ and }\|S(y - \La f)\|^2 \le \eta^2 \}.
\end{align*}
In the spirit of the S-procedure, the latter constraint is satisfied if, for some $c,d \ge 0$, 
\be
\label{CstSProc}
 \|Q(f) -z \|^2  - \gamma 
 \le c( \|Rf\|^2 - \eps^2  ) + d ( \|S(y - \La f)\|^2 - \eta^2 )
 \qquad \mbox{for all } f \in H.
\ee
Thus, fixing such $c,d \ge 0$, we obtain
$$
 {\sf sv}(z) \le  \inf_\gamma  \gamma \quad \mbox{s.to  the constraint \eqref{CstSProc}}.
$$
Reparametrizing $f \in H$ as $f = f_{c,d} + h$ with variable $h \in H$,
\eqref{CstSProc} becomes 
\begin{align*}
 \|Q(f_{c,d}) -z \|^2 - \gamma & + 2 \langle Q(f_{c,d}) -z, Qh \rangle  + \|Qh\|^2\\
 & \le  c( \|Rf_{c,d}\|^2 - \eps^2  ) + d ( \|S(y - \La f_{c,d})\|^2 - \eta^2 )\\
 & + 2 (c \langle R f_{c,d},  Rh \rangle + d \langle S(\La f_{c,d}-y), S \La h \rangle )\\
 & + c \|Rh \|^2 + d \|S \La h \|^2  
 \qquad \qquad \mbox{for all }  h \in H.
\end{align*}
The linear term in the right-hand side is, 
according to \eqref{CharaReg},
$$
2 \, \big\langle \,  c R^* R f_{c,d} + d \La^* S^* S (\La f_{c,d} -y), \, h \, \big\rangle  = 0.
$$
Therefore, 
for any $c,d \ge 0$,
we arrive at
\begin{align*}
 {\sf sv}(z) \le  & \inf_{\gamma \ge 0} \quad \gamma
 \\ & \mbox{s.to} \quad 
 \|Q(f_{c,d}) -z \|^2 + 2 \langle Q(f_{c,d}) -z, Qh \rangle \\
&  \phantom{\mbox{s.to} \quad }
\le   \gamma - c( \eps^2 -  \|Rf_{c,d}\|^2  ) - d ( \eta^2 - \|S(y - \La f_{c,d})\|^2  )\\
&  \phantom{\mbox{s.to} \quad }
 + c \|Rh \|^2 + d \|S \La h \|^2 - \|Qh\|^2  
 \qquad \quad \mbox{for all }  h \in H.
\end{align*}
Taking the infimum over $c,d \ge 0$,
we recognize the desired inequality ${\sf sv}(z) \le {\sf ub}(z)$.
Note that this inequality turns into an equality in case where the S-procedure is exact---in particular,
if the scalar field is $\bC$. 
The subsequent steps (ii)-(iii) would then establish that $Q(f_{\wt{c},\wt{d}})$ is the Chebyshev center.

\vspace{5mm}
\noindent
\ul{Justification of (ii): Part 1.}
Here, we prove that $\inf_{z \in Z} {\sf ub}(z) \ge {\sf lub}$.
Towards this end, for $z \in Z$, 
we consider the constraint \eqref{CstUB} in the defining expression of ${\sf ub}(z)$,
which we write succinctly as
$$
{\rm LHS}(h) \ge  \|Q(f_{c,d}) -z \|^2 + 2 \langle Q(f_{c,d}) -z, Qh \rangle
 \qquad \mbox{for all }  h \in H.
$$
Since ${\rm LHS}(-h) = {\rm LHS}(h)$,
averaging the above inequality for $h$ and $-h$  leads to
${\rm LHS}(h) \ge  \|Q(f_{c,d}) -z \|^2$
for all $h \in H$, and hence to ${\rm LHS}(h) \ge 0$ for all $h \in H$.
Having loosened the constraint, 
we deduce that
$$
{\sf ub}(z) \ge \inf_{\gamma,c,d \ge 0} \; \gamma
\qquad \mbox{s.to} \quad {\rm LHS}(h) \ge 0 \quad \mbox{for all }  h \in H.
$$ 
Taking the explicit form of ${\rm LHS}(h)$ into account, 
we see that the above constraint decouples as
$$
\gamma - c( \eps^2 -  \|Rf_{c,d}\|^2  ) - d ( \eta^2 - \|S(y - \La f_{c,d})\|^2  ) \ge 0
$$
and
$$
c \|Rh \|^2 + d \|S \La h \|^2 - \|Qh\|^2 \ge 0
\qquad \mbox{for all } h \in H.  
$$ 
The former reads $\gamma \ge {\sf Obj}$ and the latter is the constraint \eqref{CstLUB}.
Thus we arrive at
$$
{\sf ub}(z) \ge \inf_{\gamma,c,d \ge 0} 
\left\{ \gamma \quad  \mbox{s.to}  \quad \gamma \ge {\sf Obj} \; \mbox{ and } \; \eqref{CstLUB} \right\}
= \inf_{c,d \ge 0} \left\{ {\sf Obj} \quad \mbox{s.to} \quad \eqref{CstLUB} \right\}.
$$
This is the desired inequality ${\sf ub}(z) \ge {\sf lub}$, valid for any $z \in Z$.

\vspace{5mm}
\noindent
\ul{Justification of (iii).}
In view of
$$
c \|Rh \|^2 + d \|S \La h \|^2 - \|Qh\|^2 
 = \big\langle \, (c R^* R + d \La^* S^* S \La - Q^* Q)h, \,h \,  \big\rangle,
$$
we instantly see that the constraint \eqref{CstLUB} is equivalent to 
$c R^* R + d \La^* S^* S \La - Q^* Q \succeq 0$.
Therefore, the program defining ${\sf lub}$ is equivalent to
\be
\label{ObjEq1}
\minimize{c,d \ge 0} \; {\sf Obj}
\qquad \mbox{s.to} \quad
c R^* R + d \La^* S^* S \La  \succeq Q^* Q.
\ee
We now transform ${\sf Obj}$ by observing that
\begin{align*}
 c \eps^2 +& d \eta^2 - {\sf Obj}
 = c \|R f_{c,d}\|^2   + d \|S(\La f_{c,d} - y) \|^ 2\\
& = c \langle R^* R f_{c,d}, f_{c,d} \rangle 
+ d \langle S^* S(\La f_{c,d} - y), \La f_{c,d} - y \rangle\\
& = \langle c R^* R f_{c,d}+ d \La^* S^* S (\La f_{c,d} - y), f_{c,d} \rangle
- d \langle S^* S \La f_{c,d}, y \rangle + d \langle S^* S y, y \rangle \\
& =    - d \langle S^* S \La f_{c,d}, y \rangle + d \langle S^* S y, y \rangle,
\end{align*}
where the last step made use of the characterization \eqref{CharaReg}.
It follows that
\begin{align}
\nonumber
{\sf Obj} & = c \eps^2 + d \eta^2 -  d \| S y \|^2 +  d \langle S^* S \La f_{c,d}, y \rangle\\
\label{NewExpObj}
& = \inf_t  \; c \eps^2 + d \eta^2 -  d \| S y \|^2 +  t
\quad \mbox{s.to  } \quad t \ge d \langle S^* S \La f_{c,d}, y \rangle.
\end{align} 
According to \eqref{ExprReg},
the latter inequality can be written as
$$
t \ge d y^* S^* S \La f_{c,d} 
= (d y^* S^* S \La ) [c R^*R + d \La^* S^* S \La]^{-1 } (d\La^* S^* Sy),
$$
or equivalently as the positive semidefiniteness of a Schur complement, 
namely as
\be
\label{SchurCst}
\bbmx
c R^* R + d \La^* S^* S \La & \vline &  d \La^* S^* S y\\
\hline
 d y^* S^* S \La  & \vline & t
\ebmx
\succeq 0.
\ee
Substituting \eqref{NewExpObj} into \eqref{ObjEq1} while imposing the additional constraint \eqref{SchurCst} shows that the program defining ${\sf lub}$ is indeed equivalent to \eqref{SDPCheCenQ}.

\vspace{5mm}
\noindent
\ul{Justification of (ii): Part 2.}
It now remains to prove that ${\sf ub}(Q(f_{\wt{c},\wt{d}})) \le {\sf lub}$,
where we recall that $\wt{c},\wt{d},\wt{t}$ represent minimizers of \eqref{SDPCheCenQ}.
By (iii), this also means that $\wt{c},\wt{d}$ are minimizers of the problem defining ${\sf lub}$.
Thus, the feasibility constraint \eqref{CstLUB} is met for $c=\wt{c}$ and $d= \wt{d}$,
so choosing $\gamma = \wt{c}( \eps^2 -  \|Rf_{\wt{c},\wt{d}}\|^2  ) + \wt{d} ( \eta^2 - \|S(y - \La f_{\wt{c},\wt{d}})\|^2  )$,
we see that the constraint \eqref{CstUB} associated to ${\sf ub}(z)$ is met with $z = Q(f_{\wt{c},\wt{d}})$.
We deduce that ${\sf ub}(Q(f_{\wt{c},\wt{d}})) \le 
\gamma = \wt{c}( \eps^2 -  \|Rf_{\wt{c},\wt{d}}\|^2  ) + \wt{d} ( \eta^2 - \|S(y - \La f_{\wt{c},\wt{d}})\|^2  )$,
which is the minimal value of ${\sf Obj}$
under the constraint \eqref{CstUB}.
In other words, we have shown that ${\sf ub}(Q(f_{\wt{c},\wt{d}})) \le {\sf lub}$, as desired.
\end{proof}

\section{Exactness of the S-Procedure: Proof of Theorem \ref{ThmMain}}

Our goal in this section is to show that the overestimation of Theorem \ref{ThmQ} becomes an exact evaluation if $Q=\Id_H$, $P$ is an orthogonal projection, $S = \Id_{m}$, and $\La \La^* = \Id_m$,
thus proving Theorem \ref{ThmMain}.
We rely on duality in semidefinite programming.
Retaining full generality for the moment,
our primal program is the rewriting of \eqref{SDPCheCenQ} in the form
\begin{align}
\label{Primal}
{\sf lub} \, = & \min_{\substack{c,d \ge 0 \\ t \in \bR}}  \;
c \eps^2 + d ( \eta^2 -  \|Sy\|^2) + t\\
\nonumber
 & \;  \mbox{s.to}  \;   
 M_{c,d,t} :=
\bbmx
c R^* R + d \La^* S^* S \La  -  Q^* Q & \vline & 0 & \vline & 0\\
\hline
0 & \vline & c R^* R + d \La^* S^* S \La & \vline &  d \La^* S^* S y\\
\hline
0 & \vline & d y^* S^* S \La  & \vline & t
\ebmx
\succeq 0.
\end{align}
According to e.g. \cite[Example 5.11]{BoyVan},
its dual program reads
\begin{align}
\label{Dual}
{\sf lub}' & = \max_{X \succeq 0}
\; \tr \left( \bbmx Q^* Q & \vline & 0 & \vline & 0\\
\hline 0 & \vline &  0 & \vline & 0\\
\hline 0 & \vline & 0 & \vline & 0 \ebmx X \right)
\quad \mbox{s.to} \quad 
\tr \left( \bbmx R^* R & \vline & 0 & \vline & 0\\
\hline 0 & \vline & R^* R & \vline & 0 \\
\hline 0 & \vline & 0 & \vline & 0
\ebmx X \right) = \eps^2, \\
\nonumber
 & \qquad  \, \, \,
\tr \left( \bbmx \La^* S^* S \La & \vline & 0 & \vline & 0\\
\hline 0 & \vline & \La^* S^* S \La & \vline &  \La^* S^* S y\\
\hline 0 & \vline &  y^* S^* S \La & \vline & 0 
\ebmx X \right) = \eta^2 - \|S y\|^2,
\;
\tr \left( \bbmx 0 & \vline & 0 & \vline & 0\\
\hline 0 & \vline & 0 & \vline & 0\\
\hline 0 & \vline & 0 & \vline & 1
 \ebmx X \right)  = 1.
\end{align}
It is well known (and easy to verify)
that the inequality ${\sf lub} \ge {\sf lub}'$ always holds.
Furthermore,
if $\wh{c},\wh{d},\wh{t}$ are feasible for \eqref{Primal} and if $\wh{X}$ is feasible for \eqref{Dual},
while $\tr( M_{\wh{c},\wh{d},\wh{t}}  \wh{X}) = 0$,
then the equality ${\sf lub} = {\sf lub}'$ actually holds,
and in fact
$$
{\sf lub} = \wh{c} \eps^2 + \wh{d} ( \eta^2 -  \|Sy\|^2) + \wh{t}
= \tr \left( \bbmx Q^* Q & \vline & 0 & \vline & 0\\
\hline 0 & \vline &  0 & \vline & 0\\
\hline 0 & \vline & 0 & \vline & 0 \ebmx \wh{X} \right)
= {\sf lub}'.
$$
This simply is a consequence of equalities throughout the chain of inequalities
\begin{align*}
0 & \le {\sf lub} - {\sf lub}'
\le \wh{c} \eps^2 + \wh{d} ( \eta^2 -  \|Sy\|^2) + t
- \tr \left( {\small \bbmx Q^* Q & \vline & 0 & \vline & 0\\
\hline 0 & \vline &  0 & \vline & 0\\
\hline 0 & \vline & 0 & \vline & 0 \ebmx} \wh{X} \right)\\
& = \wh{c} \tr \left( {\small \bbmx R^* R & \vline & 0 & \vline & 0\\
\hline 0 & \vline & R^* R & \vline & 0 \\
\hline 0 & \vline & 0 & \vline & 0
\ebmx } \wh{X} \right)
+ \wh{d} \tr \left( {\small \bbmx \La^* S^* S \La & \vline & 0 & \vline & 0\\
\hline 0 & \vline & \La^* S^* S \La & \vline &  \La^* S^* S y\\
\hline 0 & \vline &  y^* S^* S \La & \vline & 0 
\ebmx } \wh{X} \right)
+ \wh{t} \tr \left( {\small \bbmx 0 & \vline & 0 & \vline & 0\\
\hline 0 & \vline & 0 & \vline & 0\\
\hline 0 & \vline & 0 & \vline & 1
 \ebmx } \wh{X} \right)\\
& = \tr( M_{\wh{c},\wh{d},\wh{t}}  \wh{X}) = 0.
\end{align*}
With the aim of choosing such suitable $\wh{c},\wh{d},\wh{t}$ and $\wh{X}$,
it is now time to lose generality and consider our specific situation where
$Q=\Id_H$, $P$ is an orthogonal projection, $S = \Id_{m}$, and $\La \La^* = \Id_m$.
As pointed out in Section \ref{SecContri},
this situation was settled in~\cite{FouLia}.
There, we showed that the set $\cK_{y,\cE} = \{ f \in H: \|Pf\| \le \eps \mbox{ and } \|\La f - y\| \le \eta\}$ admits $f_\sharp  \in H$ as its Chebyshev center
as soon as one can find $c_\sharp, d_\sharp \ge 0$ and $h_\sharp \in H$
 such that
\begin{enumerate}
\item[(a)] $\displaystyle{c_\sharp P + d_\sharp \La^* \La \succeq \Id,}$

\item[(b)] $\displaystyle{c_\sharp P f_\sharp + d_\sharp \La^* (\La f_\sharp -y) 
+ (c_\sharp P + d_\sharp \La^* \La) h_\sharp = h_\sharp,}$

\item[(c)] $\displaystyle{
\langle P f_\sharp, h_\sharp \rangle = 0,
\qquad \langle \La^*  (\La f_\sharp - y), h_\sharp \rangle = 0, 
}$

\item[(d)] $\displaystyle{
\|P f_\sharp + P h_\sharp \|^2 = \eps^2,
\qquad \| \La f_\sharp - y +  \La h_\sharp\|^2 = \eta^2.
}$
\end{enumerate}
These four sufficient conditions were verified
for $f_\sharp = f_{\tau_\sharp}$, 
where $\tau_\sharp \in (0,1)$ was selected as the solution to the implicit equation \eqref{ImpEq}
and where $f_{\tau_\sharp}$ was selected as in \eqref{RegSpe2} with $\tau = \tau_\sharp$. 
We also made the choices 
$c_\sharp = (1-\tau_\sharp)/\la_{\min}$ and
$d_\sharp = \tau_\sharp/\la_{\min}$,
where $\la_{\min}$ was the smallest eigenvalue of $(1-\tau_\sharp) P + \tau_\sharp \La^* \La$.
Finally, we took $h_\sharp$ as an associated eigenvector, so that $(c_\sharp P + d_\sharp \La^* \La)h_\sharp = h_\sharp$.
It was normalized to satisfy (d)---note that
it is the specific choice of $\tau_\sharp$ that made it possible to satisfy both equalities in (d).
Keeping these recollections in mind, 
we now set
$$
\wh{c} := c_\sharp,
\quad
\wh{d} := d_\sharp,
\quad
\wh{t} = d_\sharp \langle \La f_\sharp, y \rangle,
\qquad \mbox{and} \quad
\wh{X} = \bbmx
h_\sharp h_\sharp^* & \vline & 0 & \vline & 0\\
\hline 0 & \vline & f_\sharp f_\sharp^* & \vline & -f_\sharp\\
\hline 0 & \vline & -f_\sharp^* & \vline & 1 
\ebmx \succeq 0.
$$ 
The feasibility of $\wh{c},\wh{d},\wh{t}$ for \eqref{Primal}
follows from  (a),
combined with the inequality $\wh{t} \ge d_\sharp \langle \La f_\sharp, y \rangle$
reformulated via the Schur complement as in the justification of~(iii)---note that $f_\sharp = f_{c_\sharp,d_\sharp}$.
The feasibility of $\wh{X}$ for \eqref{Dual} is a consequence of (c) and (d):
while the third part of the constraint is obvious,
the first two parts require some work.
For the first part, we observe that
\begin{align*}
\tr & \left( \bbmx P & \vline & 0 & \vline & 0\\
\hline 0 & \vline & P & \vline & 0 \\
\hline 0 & \vline & 0 & \vline & 0
\ebmx \wh{X} \right) 
 = \tr (P h_\sharp h_\sharp^*) + \tr(P f_\sharp f_\sharp^*)
 = \|P h_\sharp \|^2 +  \|P f_\sharp \|^2\\
& \qquad  \underset{(c)}{=} \|P h_\sharp  + P f_\sharp \|^2
\underset{(d)}{=} \eps^2.
\end{align*} 
For the second part, we observe that 
\begin{align*}
\tr & \left( \bbmx \La^* \La & \vline & 0 & \vline & 0\\
\hline 0 & \vline & \La^* \La & \vline &  \La^* y\\
\hline 0 & \vline &   y^*  \La & \vline & 0 
\ebmx \wh{X} \right) 
 = \tr(\La^* \La h_\sharp h_\sharp^*) + \tr(\La^* \La f_\sharp f_\sharp^*)  - 2 \langle \La^* y, f_\sharp \rangle \\
& \qquad \qquad = \|\La h_\sharp\|^2 + \|\La f_\sharp\|^2 - 2 \langle y, \La f_\sharp \rangle
=   \|\La h_\sharp\|^2 + \|\La f_\sharp - y\|^2 - \|y\|^2 \\
& \qquad \qquad \underset{(c)}{=} \|\La f_\sharp - y + \La h_\sharp\|^2 - \|y\|^2
\underset{(d)}{= } \eta^2 - \|y\|^2.
\end{align*}
Having verified the feasibility conditions,
we turn to the complementary slackness condition.
Using the eigenvalue property of $h_\sharp$
and the characterization of $f_\sharp$
(which together accounted for (b)),
we obtain
\begin{align*}
\tr (M_{\wh{c},\wh{d},\wh{t}} \wh{X})
& = \tr( (\wh{c} P + \wh{d} \La^* \La - \Id)h_\sharp h_\sharp^* )
+ \tr( (\wh{c} P + \wh{d} \La^* \La )f_\sharp f_\sharp^* )
- 2 \wh{d} \langle \La^* y,f_\sharp \rangle  + \wh{t}\\
& = \langle (c_\sharp P + d_{\sharp} \La^* \La - \Id)h_\sharp, h_\sharp \rangle 
+ \langle ( c_\sharp P + d_\sharp \La^*( \La f_\sharp - y), f_\sharp \rangle\\
& = 0,
\end{align*}
as expected.
At this point, we are guaranteed that
\begin{align*}
{\sf lub} 
& = {\sf lub}' 
= \tr \left( \bbmx \Id_H & \vline & 0 & \vline & 0\\
\hline 0 & \vline &  0 & \vline & 0\\
\hline 0 & \vline & 0 & \vline & 0 \ebmx \wh{X} \right)
= \tr(h_\sharp  h_\sharp^*)
= \|h_\sharp\|^2 = \|f_\sharp + h_\sharp - f_\sharp \|^2 \\
& \underset{(d)}{\le} \sup_{f \in H} \; \|f - f_\sharp\|^2
\quad \mbox{s.to} \quad
\|Pf\|^2 \le \eps^2 \mbox{ and } \|\La f - y\|^2 \le \eta^2\\
& = {\sf sv}(f_\sharp) = \inf_{g \in H} {\sf sv}(g),
\end{align*}
where the last equality expresses the result of \cite{FouLia} that $f_\sharp$ is the Chebyshev.
Now, with $f_{\wt{c},\wt{d}}$ denoting the candidate Chebyshev center from Theorem \ref{ThmQ},
we recall that ${\sf sv}(f_{\wt{c},\wt{d}}) 
\le {\sf ub}(f_{\wt{c},\wt{d}}) = {\sf lub}$.
Thus, we finally derive that 
$$
{\sf sv}(f_{\wt{c},\wt{d}})  \le  \inf_{g \in H} {\sf sv}(g),
$$
i.e., that $f_{\wt{c},\wt{d}}$ is a Chebyshev center of $\{ f \in H: \|P f \| \le \eps \mbox{ and } \|y-\La f\| \le~\eta~\}$---actually, by uniqueness, it is the Chebyshev center.
This concludes the full proof of Theorem \ref{ThmMain}.

\section{Some Open Questions}

This work, as well as the previous work \cite{FouLia},
considered the model set $\cK$ and the uncertainty set $\cE$ given by
\be
\label{Mod+Unc}
\cK  = \{ f \in H: \|Pf\|_H \le \eps\}
\qquad \mbox{and} \qquad
\cE = \{ e \in \bR^m: \|e\|_2 \le \eps \}.
\ee
For the quantity of interest $Q = \Id_H$,
they settled the Chebyshev center issue under the specific conditions that $P$ is an orthogonal projector and that $\La \La^* = \Id_m$.
Without these conditions, however,
there are several issues that remain unsettled.
Here is a list of questions,
all concerning the model set and the uncertainty set from \eqref{Mod+Unc},
to which we do not know the answers---we cannot even foretell if any of them are likely to be  easily solved...

\begin{enumerate}
\item {\em Does the Chebyshev center always belong to $\cV + {\rm ran}(\La^*)$?}

Under the the specific conditions that $P$ is an orthogonal projector and that $\La \La^* = \Id_m$,
we showed in \cite{FouLia} that the solution $f_\tau$ to \eqref{RegSpe2} depends affinely on $\tau \in [0,1]$.
Since $f_0 \in \cV + {\rm ran}(\La^*)$ and $f_1 \in \cV$,
any $f_\tau$ belongs to $\cV + {\rm ran}(\La^*)$ and in particular the Chebyshev center does. 
Can this be proved directly without the specific conditions, even if the Hilbert space $H$ is not finite-dimensional?
An affirmative answer would allow one to deal with the infinite-dimensional setting by restricting the problem to the finite-dimensional space $\cV + {\rm ran}(\La^*)$.\vspace{2mm}

\item {\em Is the Chebyshev center always obtained via regularization?}

From \cite{FouLia},
we know that the answer is affirmative under the specific conditions mentioned above,
since the Chebyshev center equals the regularizer $f_{\tau_\sharp}$ with parameter $\tau_\sharp \in (0,1)$ satisfying the implicit equation \eqref{ImpEq}.
Without the specific conditions,
we also know from \cite{BecEld}
that the answer is affirmative if one considered the complex setting rather than the real setting.\vspace{2mm}

\item {\em What is the Chebyshev center relative to a quantity of interest $Q \not= \Id_H$?}

Here, we do not have any partial answers in the real setting,
even under the specific conditions above.
Indeed, the result of \cite{FouLia} only considered the case $Q = \Id_H$
and initial attempts to extend it to the case $Q \not= \Id_H$ were not conclusive.
An obvious guess (valid in the complex setting) is that the Chebyshev center equals $Q(f_{\wt{c},\wt{d}})$,
where $f_{\wt{c},\wt{d}}$ is the regularizer obtained from the relaxation of Theorem \ref{ThmQ}.\vspace{2mm}

\item {\em Does relaxation always lead to near-optimality?}

Should one fail to find the Chebyshev center,
one may be content with the statement that there exists a parameter $\tau_\natural \in (0,1)$ such that the regularizer $f_{\tau_\natural}$ admits a near-minimal local worst-case error, in the sense that
$$
\sup_{\substack{f \in \cK\\ y - \La f \in \cE}} \|Q(f) - Q(f_{\tau_\natural}) \|
\le C \times
\inf_{z \in Z}  \sup_{\substack{f \in \cK\\ y - \La f \in \cE}} \|Q(f) - z \|
$$
for some constant $C \ge 1$.
This was proved in \cite{FLV}
under no specific conditions.
We~wonder if the same holds for $f_{\wt{c},\wt{d}}$ instead of $f_{\tau_{\natural}}$.

\end{enumerate}

\section{Some Other Optimal Recovery Results Obtained at CAMDA}
\label{SecRes}

As alluded to in the introduction,
investigations in the field of optimal recovery constituted one of the pilot research projects during the academic year 2022-23, i.e., the initial year of the Center for Approximation and Mathematical Data Analytics, better known as CAMDA.
Below is a synopsis of results generated during this time.

\begin{itemize}

\item The article \cite{FLV} considered a scenario of optimal recovery from inaccurate data very close to the present one,
the only difference being that it targeted  a global worst-case error akin to \eqref{gwce} rather than a local worst-case error akin to \eqref{lwce}.
In spite of a focus on graph signals
(i.e., functions defined on the set of vertices of a graph),
it settled the global optimality question in the general setting of the model and uncertainty sets from \eqref{Mod+Unc}
without any specific conditions on $P$ and $\La$ 
and for an arbitrary linear map $Q: H \to Z$ as quantity of interest.
It established that the a globally optimal recovery map is given by the linear map $Q \circ \Delta_{\tau_\sharp}: \bR^m \to Z$, where
$$
\Delta_{\tau_\sharp}
: y \in \bR^m
\mapsto 
\Big[ \underset{f \in H}{\argmin} \; 
(1-\tau_{\sharp}) \|Pf\|^2 + \tau_\sharp \|y - \La f \|^2 \Big] \in H
$$
and where the parameter $\tau_\sharp \in (0,1)$ is the ratio $\tau_\sharp = d_\sharp/(c_\sharp + d_\sharp)$ featuring the solutions $c_\sharp,d_\sharp \ge 0$ to the semidefinite program
$$
\minimize{c,d \ge 0} \; c \eps^2 + d \eta^2
\qquad \mbox{s.to }
c P^* P + d \La^* \La \succeq Q^* Q.
$$
Precisely, this means that
$$
\sup_{\substack{f \in \cK \\ e \in \cE}}
\|Q(f) - (Q \circ \Delta_{\tau_\sharp})(\La f +e) \|
= \inf_{\Delta: \bR^m \to Z}
\sup_{\substack{f \in \cK \\ e \in \cE}}
\|Q(f) - \Delta(\La f +e) \|. 
$$\vspace{2mm}

\item The article \cite{FouPao} viewed the observation errors differently,
since $e \in \bR^m$ was considered to be a random vector there.
It advocated for the global error
$$
{\rm ge}^{\rm or}_p(\Delta)
:= \bigg( \bE \Big[ \sup_{f \in \cK} \|Q(f) - \Delta(\La f + e) \|_Z^p \Big] \bigg)^{1/p},
\qquad p \ge 1,
$$
to be used in optimal recovery, as opposed to another notion appearing more often in statistical estimation, namely
$$
{\rm ge}^{\rm se}_p(\Delta)
:= \bigg( \sup_{f \in \cK} \bE \Big[  \|Q(f) - \Delta(\La f + e) \|_Z^p \Big] \bigg)^{1/p},
\qquad p \ge 1.
$$
In an arbitrary Banach space $F$,
if $Q \in F^*$ is a linear functional,
if the model set $\cK$ is symmetric and convex,
and if the random vector $e \in \bR^m$ is merely log-concave,
it was proved that linear maps are near optimal relative to ${\rm ge}^{\rm or}$, 
in the sense that 
\be
\inf_{\Delta: \bR^m \to \bR \,{\rm linear}}
{\rm ge}^{\rm or}_p(\Delta^{\rm lin})
\le C_p \times \inf_{\Delta: \bR^m \to \bR} {\rm ge}^{\rm or}_p(\Delta)
\ee
for some constant $C_p$ depending on $p$.
This extends the seminal result of \cite{Don},
which dealt with ${\rm ge}^{\rm se}$ instead of ${\rm ge}^{\rm or}$ and considered only gaussian vectors $e \in \bR^m$.\vspace{2mm}

\item The article \cite{Ful} did not take place in a Hilbert space $H$ but in the Banach space of continuous functions on a compact set $\cX$,
i.e., $F = C(\cX)$.
It targeted,
again from a global optimality perspective,
the full recovery problem ($Q=\Id_{C(\cX)}$)
relative to a model set based on approximation capabilities relative to some linear subspace $\cV \inc C(\cX)$ and to a parameter $\eps > 0$, namely 
$$
\cK_\cV = \{ f \in C(\cX):
 {\rm dist}_{C(\cX)} (f,\cV) \le \eps \}.
$$
The observational data were assumed to be acquired as accurate evaluations at points $x^{(1)},\ldots,x^{(m)} \in \cX$.
In this situation,
it was known (see \cite{DFPW})
how to optimally estimate point evaluations $\delta_x \in C(\cX)^*$ at each $x \in \cX$,
and how to subsequently produce an optimal recovery map for $Q=\Id_{C(\cX)}$---a linear one, to boot.
However, the construction was not practical.
A practical construction,
which amounts to solving about $m$ standard-form linear programs via the simplex method,
was conceived in \cite{Ful}
under the proviso that $\cV$ is a Chebyshev space.
Although this is a strong proviso---by Mairhuber--Curtis theorem, it essentially excludes multivariate functions---it brings hope that genuinely optimal recovery maps in $C(\cX)$ can be computationally constructed in specific instances.\vspace{2mm}

\item 
The article \cite{BBCDDP} studied the numerical solutions to partial differential equations.
With $\Omega$ being a bounded Lipschitz domain in $\bR^d$ with $d=2$ or $d=3$, 
it considered the elliptic problem
\begin{equation*}
    -\Delta u = f \quad \mbox{in } \Omega, 
    \qquad \qquad u=g \quad \mbox{on } \partial\Omega, 
\end{equation*}
where $f \in H^{-1}(\Omega)$ is known
but $g \in H^{1/2}(\partial\Omega)$
is not known
and when linear observations $y = \La u$ made on the solution $u \in H^1(\Omega)$ are available.
Finding an approximant $\wh{u}$ is viewed as an optimal recovery problem relative to the model set 
$$
\cK_s = \{ 
u \in H^1(\Omega): -\Delta u = f
\mbox{ and } 
u_{| \partial \Omega} = g 
\mbox{ for some } g \in B_{H^s(\partial \Omega)}
\},
\quad
s > 1/2.
$$
Based on minimum-norm interpolation (aka spline algorithm),
an implementable procedure was conceived to generate an output to $\wh{u}$ which is locally near-optimal,
in the sense that 
$\sup\{ \|u - \wh{u} \|_{H^1(\Omega)} \colon
u \in \cK_s, \La u = y\}$ is a most a multiplicative constant away from the Chebyshev radius of the set $\{ u \in \cK_s \colon \La u = y \}$. 
\end{itemize}

\begin{acknowledgement}
S. F. is partially supported by grants from the NSF (DMS-2053172) and from the ONR (N00014-20-1-2787).
\end{acknowledgement}



\begin{thebibliography}{99.}%
%
%

\bibitem{Kol}
Tikhomirov, V. M.:
A. N. Kolmogorov and Approximation Theory. Russian Mathematical Surveys, \textbf{44}, 101 (1989).

\bibitem{MicRiv}
Micchelli, C. A.,  Rivlin, T. J.:
A survey of optimal recovery.
In: Optimal estimation in approximation theory.
pp.1--54. (1977)

\bibitem{NovWoz}
Novak, E.,  Wozniakowski, H. : Tractability of Multivariate Problems, Volume I.
European Mathematical Society, Z\"urich (2010)

\bibitem{BookDS}
Foucart, S.:
Mathematical Pictures at a Data Science Exhibition. 
Cambridge University Press (2022).

\bibitem{Gar}
Garkavi., A. L. : 
On the optimal net and best cross-section of a set in a normed space. 
Izvestiya Rossiiskoi Akademii Nauk. Seriya Matematicheskaya \textbf{26},  87--106 (1962)

\bibitem{MelMic}
Melkman, A. A. , Micchelli,  C. A. : 
Optimal estimation of linear operators in Hilbert spaces from inaccurate data. 
SIAM Journal on Numerical Analysis \textbf{16}, 87--105 (1979)

\bibitem{Mic}
Micchelli, C. A. : 
Optimal estimation of linear operators from inaccurate data: a second look. 
Numerical Algorithms \textbf{5}, 375--390 (1993)

\bibitem{BecEld}
Beck, A., Eldar, Y.C.:
Regularization in regression with bounded noise: a Chebyshev center approach. 
SIAM Journal on Matrix Analysis and Applications \textbf{29}, 606--625 (2007)

\bibitem{BCDDPW}
Binev, P. , Cohen, A., Dahmen, W.,  DeVore, R. , Petrova, G.,  Wojtaszczyk, P.:
Data assimilation in reduced modeling. 
SIAM/ASA Journal on Uncertainty Quantification \textbf{5},  1--29 (2017)

\bibitem{FouLia}
Foucart, S. , Chunyang,  L. :
Optimal recovery from inaccurate data in Hilbert spaces: regularize, but what of the parameter? Constructive Approximation \textbf{57}, 489--520 (2023)

\bibitem{PolTer}
P\'olik. I., Terlaky, T.:
A survey of the S-lemma. 
SIAM Review \textbf{49},  371--418 (2007)

\bibitem{Yak}
Yakubovich, V. A. : 
S-procedure in nonlinear control theory.
Vestnik Leningrad. Univ. \textbf{1}, 62--77  (1971)

\bibitem{BecEld2}
Beck, A. , Eldar, Y. C.:
Strong duality in nonconvex quadratic optimization with two quadratic constraints.
SIAM J. Optim. \textbf{17},  844--860  (2006)

\bibitem{Pol}
Polyak,  B. T.:
Convexity of quadratic transformations and its use in control and optimization.
J. Optim. Theory Appl. \textbf{99},  553--583 (1998)

\bibitem{BoyVan}
Boyd, S. P. ,   Vandenberghe, L. :
Convex Optimization. 
Cambridge University Press (2004)

\bibitem{FLV}
Foucart, S., Liao, C., and Veldt, N.: 
On the optimal recovery of graph signals. 
In: International Conference on Sampling Theory and Applications (SampTA), 2023.

\bibitem{FouPao}
Foucart, S. and Paouris, G.:
Near-optimal estimation of linear functionals with log-concave observation errors. 
Information and Inference. To appear.

\bibitem{Don}
Donoho, D. L.: 
Statistical estimation and optimal recovery. The Annals of Statistics \textbf{22}, 238--270 (1994).

\bibitem{DFPW}
DeVore, R.,  Foucart,  S., Petrova, G.,  Wojtaszczyk, P.:
Computing a quantity of interest from observational data.
Constructive Approximation \textbf{49}, 461--508 (2019)

\bibitem{Ful}
Foucart, S.: 
Full recovery from point values: an optimal algorithm for Chebyshev approximability prior.
Advances in Computational Mathematics \textbf{49},  57 (2023)

\bibitem{BBCDDP}
Binev, P. , Bonito, A., Cohen, A., Dahmen, W.,  DeVore, R. , Petrova, G.:
Solving PDEs with incomplete information. 
arXiv:2301.05540 (2023).







\end{thebibliography}
\end{document}